\documentclass[11pt]{amsart}

\usepackage{amsfonts,amsmath,amsthm,amscd,amssymb,latexsym,mathrsfs, enumitem}
\usepackage[english]{babel}

\usepackage[margin=1in]{geometry}

\theoremstyle{plain}

\newtheorem{theorem}[subsection]{Theorem}
\newtheorem{lemma}[subsection]{Lemma}

\newtheorem{proposition}[subsection]{Proposition}
\newtheorem{corollary}[subsection]{Corollary}

\newtheorem{definition}[subsection]{Definition}

\newtheorem{question}[subsection]{Question}

\makeatletter
\@addtoreset{equation}{section}
\@addtoreset{figure}{section}
\@addtoreset{table}{section}
\makeatother

\numberwithin{equation}{section}
\newcommand\RRR{\mathbb{R}}
\newcommand\ZZZ{\mathbb{Z}}
\newcommand\NNN{\mathbb{N}}

\newcommand\Bo{\mathrm{B}_1}
\newcommand\hBo{\mathrm{hB}_1}

\newcommand\cB[1]{\Bo(#1)}
\newcommand\hcB[1]{\hBo(#1)}

\begin{document}
\title{The first homotopic Baire class of maps with values in ANR's coincides with the first Baire class}
\author{Olena Karlova}
 
\author{Sergiy Maksymenko}

\address{Yurii Fedkovych Chernivtsi National University, Chernivtsi, Ukraine}

\address{Institute of Mathematics of National Academy of Sciences, Kyiv, Ukraine}

\begin{abstract}
We introduce the first homotopic Baire class of maps as a homotopical counterpart of a usual first Baire class of maps between topological spaces and show that those classes with values in ANR spaces coincide.
\end{abstract}

\maketitle

\section{Introduction}

Let $X$ and $Y$ be topological spaces.

\begin{definition}
{\rm A map $f:X\to Y$ belongs to {\it the first Baire class} or is {\it a Baire-one map} \cite{Baire1}, if there exists a sequence $(f_n)_{n\in\omega}$ of continuous maps converging to $f$ pointwisely on $X$, that is
  \begin{gather}\label{gath:1}
    \lim_{n\to\infty} f_n(x)=f(x)
  \end{gather}
  for every $x\in X$.}
\end{definition}

%\begin{definition}
%  {\rm A map $f:X\to Y$ belongs to {\it the first Baire class} or is {\it a Baire-one map}, if there exists a sequence continuous map $H:X\times \omega\to Y$ such that
%  \begin{gather}\label{gath:1}
%    \lim_{n\to\infty} H(x,n)=f(x)
%  \end{gather}
%  for every $x\in X$.}
%\end{definition}

There are many characterizations of real-valued Baire-one functions. For example, the classical results of Baire, Banach, Lebesgue and Hausdorff (see~\cite{Baire1}, \cite{BanachS:1931}, \cite{Leb905}, \cite[Theorem 24.10]{Kechris}) state that for a real valued function $f$ defined on a Banach space $X$ the following conditions are equivalent: (i) $f$ is a Baire-one function; (ii) for every non-empty closed subset $K\subseteq X$, the restriction $f|_{K}$ has a point of continuity relative to the topology of $K$; (iii) $f$ is $F_\sigma$-measurable (i.e., for any open set $V\subseteq \mathbb R$ the preimage $f^{-1}(V)$ is of the type $F_\sigma$ in $X$).

These results were generalized by many mathematicians in different ways (see \cite{Hansell:1974},\cite{Fos},\cite{ATZ}  and the literature given there). The aim of the present paper is to introduce and study a homotopy counterpart of Baire-one maps.

\begin{definition}
  \rm
We will say that a map $f:X\to Y$ belongs to {\it the first homotopic Baire class} or is {\it a homotopical Baire-one map}, if there exists a sequence $(f_n)_{n\in\omega}$ of pairwise homotopic continuous maps $f_n:X\to Y$ such that $\lim\limits_{n\to\infty}f_n(x)=f(x)$ for every $x\in X$.

Equivalently, $f:X \to Y$ is homotopical Baire-one map if and only if there exists a homotopy $H:X \times [0,+\infty)\to Y$ such that
\begin{gather}\label{gath:2}
  \lim_{n\to\infty} H(x,n)=f(x)
\end{gather}
for every $x\in X$.
In this case $H$ will be called a {\it $\Bo$-homotopy for $f$}.
\end{definition}
The collection of all (resp. homotopic) Baire-one maps between $X$ and $Y$ will be denoted by $\cB{X,Y}$ (resp. $\hcB{X,Y}$).
Evidently,
$$
\hcB{X,Y} \ \subseteq \ \cB{X,Y}
$$
for any spaces $X$ and $Y$.
We are interested in the difference between these classes.

Our main result  (Theorem~\ref{th:X_YANR}) shows that for a topological space  $X$ and for $Y$ being ANR we have that $\hcB{X,Y}=\cB{X,Y}$.
This means that if a map $f:X \to Y$ for such spaces can be pointwise approximated by a sequence of continuous maps, then one can assume that this sequence consists of pairwise homotopic maps.

\subsection{Examples}
In general the sequence $(f_n)$ from~\eqref{gath:1} which pointwise approximates $f:X \to Y$ may consists of distinct homotopy classes.
The following example shows that every continuous self-map of the circle can be poXint-wise approximated by a sequence of continuous maps belonging to any given sequence of homotopy classes.
This illustrates that homotopy invariants are not preserved by pointwise convergence, and motivates introduction of first homotopy Baire classes.

Let $S^1 = \{z\in\mathbb C: |z|=1 \}$ be the unit circle in the complex plane also regarded as the quotient $[0,1] / \{0,1\}$.

For each $n\in \ZZZ$ and $\alpha\in(0,1)$ define the function $f_{n,\alpha}:[0,1]\to\RRR$ by
\[
f_{n,\alpha}(t) =
\begin{cases}
t & t \leq 1-\alpha, \\
n + (t-1) \frac{n-1+\alpha}{\alpha}, & t\in[1-\alpha, 1].
\end{cases}
\]
Since $f_{n,\alpha}(1)=n$ is an integer value, we get a well defined map $g_{n,\alpha}:S^1\to S^1$
\[
g_{n,\alpha}(t) = e^{2\pi i f_{n,\alpha}(t)}.
\]

Geometrically, $g_{n,\alpha}$ is fixed when the angle $t$ changes from $0$ to $1-\alpha$, and for the rest of parameters $[1-\alpha,1]$ it wraps around $S^1$ to make full $n$ rotations.
In particular, $\deg g_{n,\alpha} = n$.
The proof of the following statement is easy and we leave it to the reader.
\begin{proposition}\label{prop:ex}
Let $(\alpha_{i}) \subset (0,1)$ be any sequence such that $\lim\limits_{i\to\infty} \alpha_i = 1$, and $(n_i) \subset \ZZZ$ be any sequence of integers.
Then the sequence of maps $g_{n_i,\alpha_i}:S^1\to S^1$ converges to the identity map $\mathrm{id}_{S^1}$ pointwise on $S^1$.
Moreover, this sequence even stabilizes, i.e. for each $\beta\in(0,1)$ there exists $N\in \NNN$ such that $g_{n_i,\alpha_i}$ is fixed on $[0,\beta]$ for all $i>N$.

More generally, if $X$ and $Y$ are topological spaces and $p: X \to S^1$ and $q:S^1\to Y$ are continuous maps, then the sequences $g_{n_i,\alpha_i} \circ p:X\to S^1$ and $q \circ g_{n_i,\alpha_i}:S^1 \to Y$ pointwise converge to $p$ and $q$ respectively.
\qed
\end{proposition}

It is essential, that the statement of Proposition~\ref{prop:ex} does not depend on $(n_i)$. For instance we can choose $(n_i)$ so that each integer $n$ is presented there infinitely many times.
It is not hard to extend this example to the maps of $n$-dimensional sphere $S^n$, for $n\geq2$.

The next example shows that the classes $\hcB{X,Y}$ and $\cB{X,Y}$ are different.
\begin{proposition}\label{ex:b1_is_not_hb1}
  Let $\mathbb Q$ be the set of all rational numbers in $\mathbb R$. There exists a Baire-one function $f:\mathbb Q\to \mathbb Q$ which does not belong to $\hcB{\mathbb Q,\mathbb Q}$.
\end{proposition}

\begin{proof}
  Consider any bijection $\varphi:\mathbb N\to\mathbb Q$ and put
    \begin{gather*}
      f(x)=\tfrac{1}{n}, \mbox{\,\,\,if\,\,\,} x=\varphi(n).
    \end{gather*}
    Then Theorem 24.10 from~\cite{Kechris} implies that $f\in\cB{\mathbb Q,\mathbb Q}$.

    Assume that there exists a continuous map $H:\mathbb Q\times[0,+\infty)\to \mathbb Q$ such that~(\ref{gath:2}) holds. Since for any $x\in\mathbb Q$ the set $H(\{x\}\times [0,+\infty))$ is connected, $H|_{\{x\}\times [0,+\infty)}$ is a constant map. Then condition~(\ref{gath:2}) implies that $H|_{\{x\}\times [0,+\infty)}=f(x)$ for every $x\in\mathbb Q$. It follows that $H|_{\mathbb Q\times \{t\}}=f$ for any $t\in [0,+\infty)$. Since $f$ is everywhere discontinuous on $\mathbb Q$, we obtain  a contradiction. Hence, $f\not\in\hcB{\mathbb Q,\mathbb Q}$.
\end{proof}

\section{The simplest cases when the equality $\hcB{X,Y}=\cB{X,Y}$ holds}
We say that a topological space $Y$ is  {\it an adhesive for $X$} \cite{Karlova:Mykhaylyuk:Stable} (and denote this fact by $Y\in {\rm Ad}(X)$), whenever for any two disjoint functionally closed  sets   $A$ and $B$ in $X$ and continuous maps $f,g:X\to Y$ there exists a continuous map $h:X\to Y$ such that $h|_A=f|_A$ and $h|_B=g|_B$.

It is easy to see that $Y$ is contractible if and only if $Y\in {\rm Ad}(Y\times [0,1])$.
Moreover, the following simple fact holds true.
\begin{lemma}\label{lem:ad_contr}
Let $X$ be a topological space and $Y$ be a contractible space. Then $Y\in{\rm Ad}(X)$.
\end{lemma}
\begin{proof}
Let $\lambda:Y\times [0,1]\to Y$ be a contraction into some point $y_0\in Y$, that is $\lambda(y,0)=y$ and $\lambda(y,1)=y_0$ for all $y\in Y$.
Let also $A,B\subseteq X$ be two disjoint functionally closed sets.
Then one can find a continuous function $\varphi:X\to [0,2]$ such that $A=\varphi^{-1}(0)$ and $B=\varphi^{-1}(2)$.
Now for any two continuous maps $f,g:X\to Y$ the map $h:X \to Y$ defined by
\[
h(x)=
\begin{cases}
    \lambda(f(x),\varphi(x)), & x\in\varphi^{-1}([0,1]), \\
    \lambda(g(x),2-\varphi(x)), & x\in\varphi^{-1}((1,2])
\end{cases}
\]
has the property that $h|_A=f|_A$ and $h|_B=g|_B$.
\end{proof}

\begin{proposition}\label{prop:contra}
Let $X$ and $Y$ be topological spaces and $f\in \cB{X,Y}$.
Then in each of the following cases we have that $f\in \hcB{X,Y}$.
  \begin{enumerate}[label={\rm(\arabic*)}]
    \item\label{it1:prop:equalityOfBhB} All continuous maps $f:X \to Y$ null-homotopic, i.e. homotopic to a constant map into the same point $y\in Y$.
For instance this hold when either $X$ or $Y$ is contractible, or when $X=S^n$ is an $n$-dimensional sphere, $n\geq1$, and $Y$ is path connected and its $n$-th homotopy group vanish: $\pi_n Y = 0$.

    \item\label{it2:prop:equalityOfBhB} The image $f(X)$ is contained in a contractible subspace $T \subset Y$.

    \item\label{it3:prop:equalityOfBhB} $Y\in {\rm Ad}(X\times [0,1])$.

    \item\label{it4:prop:equalityOfBhB} The image $f(X)$ is finite and $Y$ is a path-connected Hausdorff space.
  \end{enumerate}
\end{proposition}
\begin{proof} It is easy to see that the sufficiency of conditions \ref{it1:prop:equalityOfBhB} and \ref{it2:prop:equalityOfBhB} is obvious.

\ref{it3:prop:equalityOfBhB}.
Suppose $Y\in {\rm Ad}(X\times [0,1])$ and let $H:X\times \omega\to Y$ be a map satisfying~\eqref{gath:1}.
For $n\in\omega$ define a continuous map $f_n:X\times [0,+\infty)\to Y$ by $f_n(x,t)=H(x,n)$ for $(x,t)\in X\times [0,+\infty)$.
As $Y$ is an adhesive for $X\times [0,1]$, there exists a sequence $(g_n)_{n\in\omega}$ of continuous maps $g_n:X\times [0,+\infty)\to Y$ such that $g_n|_{X\times \{n\}}=f_n|_{X\times \{n\}}$ and $g_n|_{X\times \{n+1\}}=f_{n+1}|_{X\times \{n+1\}}$ for every $n\in\omega$. It remains to put $\widetilde H(x,t)=g_n(x,t)$ if $x\in X$ and $t\in [n,n+1)$ for some $n\in\omega$.

\ref{it4:prop:equalityOfBhB}.
Assume that $Y$ is path-connected and Hausdorff and $f(X)=\{y_0,y_1,\dots,y_n\}$ is a finite set.
Then one can construct an embedded into $Y$ finite contractible $1$-dimensional CW-complex $T$ (``a topological tree'') containing $f(X)$, whence our statement will follow from~\ref{it2:prop:equalityOfBhB}.

For the proof we will use an induction on $n$.
If $n=0$, then we can set $T = f(X) = \{y_0\}$.
Suppose we have already find an embedding of a finite contractible $1$-dimensional CW-complex $T \subset Y$ such that $\{y_0,\ldots,y_{n-1}\} \subset T$.
%Since $Y$ is Hausdorff, $T$ is a closed subset of $Y$.
As $Y$ is path connected and Hausdorff, we can connect $y_n$ with some point in $T$ via an embedded path $\phi:[0,1] \to Y$ such that $\phi(0)\in Y$, $\phi(1)\in T$, and $\phi(0,1) \subset Y\setminus (T \cup \{y_n\})$.
Then $T' = T\cup \phi[0,1]$ is the required finite contractible $1$-dimensional CW-complex containing $f(X)$.
%
%Now use
%Suppose $n\geq1$.
%
%
%We claim that there exists a metrizable contractible space $D\subseteq Y$, actually a ``finite topological tree'', such that $f(X)\subseteq D$.
%Assume that $n\ge 1$.
%For all $i\in\{1,\dots,n\}$ we choose a homeomorphic embedding $\varphi_i:[0,1]\to Y$ such that $\varphi_i(0)=y_0$ and $\varphi_i(1)=y_i$.
%We put
%$$
%a_{i,j}=\sup\{t\in[0,1]:\varphi_i(t)=\varphi_j(t)\},
%$$
%where $1\le i<j\le n$, and let
%$$
%a_j=\max\limits_{1\le j\le n-1}\{a_{i,j+1}:1\le i\le j\}.
%$$
%Denote $D_1=\varphi_1([0,1])$ and define sets $D_j$ for $j\in\{2,\dots,n-1\}$ inductively:
%$$
%D_j=D_{j-1}\cup \varphi_j([a_{j-1},1]).
%$$
%Then the set
%$$
%D=D_{n-1}
%$$
%is the required one.
%
%It remains to apply condition \ref{it2:prop:equalityOfBhB}.
\end{proof}

For example, $\hcB{S^n,S^1}=\cB{S^n,S^1}$ for $n\geq2$.
%\medskip\hrule\medskip

\section{Lifting theorem for $\sigma$-discrete maps}
Let $\mathcal M_0(X)$ be the family of all functionally closed subsets of  $X$ and let $\mathcal A_0(X)$ be the family of all functionally open subsets of  $X$. For every  $\alpha\in [1,\omega_1)$ we put
\begin{gather*}
   \mathcal M_{\alpha}(X)=\Bigl\{\bigcap\limits_{n=1}^\infty A_n: A_n\in\bigcup\limits_{\beta<\alpha}\mathcal A_{\beta}(X),\,\, n=1,2,\dots\Bigr\}\,\,\,\mbox{and}\\
   \mathcal A_{\alpha}(X)=\Bigl\{\bigcup\limits_{n=1}^\infty A_n: A_n\in\bigcup\limits_{\beta<\alpha}\mathcal M_{\beta}(X),\,\, n=1,2,\dots\Bigr\}.
 \end{gather*}
Elements from the class  $\mathcal M_\alpha(X)$ belong to {\it the $\alpha$'th functionally multiplicative class} and elements from  $\mathcal A_\alpha(X)$ belong to {\it the $\alpha$'th functionally additive class} in  $X$. We say that a set is {\it functionally ambiguous of the $\alpha$'th class} if it belongs to $\mathcal M_\alpha(X)\cap\mathcal A_\alpha(X)$.

A family $\mathscr A=(A_i:i\in I)$ of subsets of a topological space $X$ is called
   {\it discrete}, if every point of $X$ has an open neighborhood which intersects at most one set from $\mathscr A$;
       {\it strongly functionally discrete} or, briefly, {\it an sfd family}, if there exists a discrete family $(U_i:i\in I)$ of functionally open subsets of $X$ such that $\overline{A_i}\subseteq U_i$ for every $i\in I$. Notice that in a metrizable space $X$ each discrete family of sets is strongly functionally discrete.

A family $\mathscr B$ of sets of a topological space $X$ is called {\it a base}  for a map $f:X\to Y$ if the preimage $f^{-1}(V)$ of any open set  $V$ in $Y$ is a union of sets from $\mathscr B$. In the case when $\mathscr B$ is a countable union of sfd families, we say that $f$ is $\sigma$-sfd discrete and write this fact as  $f\in \Sigma^s(X,Y)$. If a map $f$ has a $\sigma$-sfd base which consists of functionally ambiguous sets of the $\alpha$'th class, then we say that $f$ belongs to   the class  $\Sigma_\alpha^s(X,Y)$.

A continuous map $\varphi:Y\to X$ between topological spaces $Y$ and $X$  is {\it a weak local homeomorphism}~\cite{Karlova:Mykh:Umzh:2008} if for any point $x\in X$ there are an open neighborhood $V_x$ of $x$ and an open set $U_x\subseteq Y$ such that the restriction $\varphi|_{U_x}:U_x\to V_x$ is a homeomorphism. Moreover, we say that the space $X$ is   {\it weakly covered by $Y$}.

Let $\varphi:X\to Y$ be a weak local homeomorphism. We assign with the map $\varphi$ two families $\mathscr U$ and $\mathscr V$ of open sets in $X$ and $Y$ respectively with the following properties:
\begin{enumerate}
  \item $\mathscr V$ is covering of $Y$;

  \item for every $V\in\mathscr V$ there exits $U_V\in\mathscr U$ such that $\varphi|_{U_V}:U_V\to V$ is a homeomorphism.
\end{enumerate}
The families  $\mathscr U$ and $\mathscr V$ are said to be {\it associated with $\varphi$}.

Assume that $X$, $Y$ and $Z$ are topological spaces and $\varphi:Z\to Y$ is a weak local homeomorphism. If $\mathscr P$ is a property of maps, then we say that the triple $(X,Y,Z)$ has {\it $\mathscr P$-Lifting Property} whenever for all $f\in\mathscr P(X,Y)$ there exists $g\in\mathscr P(X,Z)$ such that $f=\varphi\circ g$.

\begin{theorem}\label{thm:lift}
  Let $X, Y, Z$ be topological spaces and $Y$ is weakly covered by   $Z$.  If $Y$ is paracompact, then $(X,Y,Z)$ has $\Sigma_\alpha^s$-Lifting Property for all $\alpha\in[1,\omega_1)$.
\end{theorem}

\begin{proof}
Let $\varphi:Z\to Y$ be a weak local homeomorphism and $\mathscr U$, $\mathscr V$ be associated families with $\varphi$ in $Z$ and $Y$, respectively.
Since $Y$ is paracompact, there exists a $\sigma$-discrete open refinement $\mathscr W$ of $\mathscr V$. For each $W\in\mathscr W$ we choose $V\in\mathscr V$ such that $W\subseteq V$. The continuity of $\varphi$ implies that the set $U_W=\varphi^{-1}(W)$ is open in $Z$ for all $W\in\mathscr W$. Moreover, $\varphi|_{U_W}:U_W\to W$ is a homeomorphism.

We consider a map $f\in\Sigma_\alpha^s(X,Y)$ and a  base $\mathscr B=\bigcup_{n\in\omega}\mathscr B_n$ for $f$ such that every $\mathscr B_n$ is strongly functionally discrete family consisting of functionally ambiguous sets of the $\alpha$'th class in $X$. Let $(\mathscr W_n:n\in\omega)$ be a sequence of discrete families of open sets in $Y$ with $\mathscr W=\bigcup_{n\in\omega}\mathscr W_n$. For every $n\in\omega$ we put $W_n=\bigcup\mathscr W_n$ and choose a subfamily $\mathscr B_{W_n}$ of $\mathscr B$ such that $B_n=f^{-1}(W_n)=\bigcup\mathscr B_{W_n}$. Notice that $B_n\in\mathcal A_\alpha(X)$ for all $n\in\omega$ and $(B_n:n\in\omega)$ is a covering of $X$. Using functional version of Reduction Theorem~\cite[Lemma 3.2]{Karlova:CMUC:2013} we choose a partition $(A_n:n\in\omega)$ of $X$ by  functionally ambiguous sets of the $\alpha$'th class with $A_n\subseteq B_n$ for all $n\in\omega$. Notice that the map $\varphi^{-1}|_{W_n}:W_n\to Z$ is continuous, since the family $\mathscr W
 _n$ is discrete. We put $C_n=f(A_n)$ and $D_n=\varphi^{-1}(C_n)$. Then the map $\psi_n=\varphi^{-1}|_{C_n}:C_n\to D_n$ is a homeomorphism for every $n\in\omega$.

Now for all $x\in X$ we define
\begin{gather*}
  g(x)=\psi_n(f(x))
\end{gather*}
if $x\in A_n$ for some $n\in\omega$.

We show that $g\in\Sigma_\alpha^s(X,Z)$. Let $\widetilde{\mathscr B}_n=(B\cap A_n:B\in\mathscr B_n)$ for all $n\in\omega$. It is easy to see that the family $\widetilde{\mathscr B}_n$ is strongly functionally discrete and consists of functionally ambiguous sets of the $\alpha$'th class. Moreover, the family $\widetilde{\mathscr B}=\bigcup_{n\in\omega}\widetilde{\mathscr B}_n$ is a $\sigma$-sfd base for $f$. Indeed, fix an open set $G\subseteq Z$. Then
\begin{gather*}
  g^{-1}(G)=\bigcup_{n\in\omega}(g^{-1}(G)\cap A_n)=\bigcup_{n\in\omega}(f^{-1}(\varphi(G\cap U_n))\cap A_n),
\end{gather*}
where $U_n=\cup\{U_{W}:W\in\mathscr W_n\}$. Since  the set $\varphi(G\cap U_n)$ is open for every $n\in\omega$, there exists a subfamily $\mathscr A_n\subseteq\mathscr B$ such that $f^{-1}(\varphi(G\cap U_n))=\bigcup\mathscr A_n$. Then $\widetilde{\mathscr A}=\bigcup_{n\in\omega}(A\cap A_n: A\in\mathscr A_n)$ is a subfamily of $\widetilde{\mathscr B}$ such that $g^{-1}(G)=\bigcup\widetilde{\mathscr A}$.

Finally, we prove that $f=\varphi\circ g$. Fix $x\in X$ and find $n\in\omega$ such that $x\in A_n$. Then the equalities
$$
\varphi(g(x))=\varphi(\psi_n(f(x)))=\varphi(\varphi^{-1}(f(x)))=f(x)
$$
finish the proof.
\end{proof}

\begin{corollary}\label{cor:cover}
Let $X$ be a topological space and $Y$ be a paracompact space which is weakly covered by a metrizable contractible locally path-connected space $Z$. Then
\[ \cB{X,Y}=\hcB{X,Y}. \]
Moreover, every $f\in\cB{X,Y}$ is a pointwise limit of a sequence of null-homotopic continuous maps.
\end{corollary}

\begin{proof}
We only need to prove the inclusion $\cB{X,Y}\subseteq h\cB{X,Y}$.
Let $f\in \cB{X,Y}$.
Then $f\in\Sigma_1^s(X,Y)$ by~\cite[Theorem 2.5]{Karlova:EJM:2015}.
Let $\varphi:Z\to Y$ be a weak local homeomorphism. According to Theorem~\ref{thm:lift} there exists a map $g\in\Sigma_1^s(X,Z)$ such that $f=\varphi\circ g$.
Since $Z$ is path-connected and locally path-connected, \cite[Theorem 4.1]{Karlova:EJM:2015} implies that $g\in {\rm B}_1(X,Z)$.
Furthermore, as $Z$ is contractible, we get from~\ref{it1:prop:equalityOfBhB} of Proposition~\ref{prop:contra} that $g\in {\rm hB}_1(X,Z)$.
Let $H_g:X\times [0,+\infty)\to Z$ be a $\Bo$-homotopy for $g$.
It is easy to see that $H_f=\varphi\circ H_g:X\times[0,+\infty)\to Y$ is a $\Bo$-homotopy for $f$.
Moreover, each map $(H_f)_t:X\times t \to Y$ is null-homotopic due to contractibility of $Z$.
\end{proof}

\section{Covering Theorem}

For a metric space $X$ by $B(a,r)$ and $B[a,r]$ we denote an open ball and a closed ball with  a center $a\in X$ and a radius $r\ge 0$.

\begin{theorem}\label{thm:cover}
 Any open path-connected subset of a convex  set in normed space is weakly covered by a contractible and locally contractible metrizable space.
\end{theorem}

\begin{proof}
  Let $E$ be a convex set in a normed  space $X$ and   $G$  be an open subset of $E$. Take an open set $H$ in $X$ such that $H\cap E=G$ and consider a covering $\mathscr H=(H_s:s\in X)$ of $H$ by open convex balls $H_s=B(a_s,r_s)$ in $X$. Put $S_0=\{s\in S:H_s\cap E\ne\emptyset)\}$, $\mathscr U=(H_s\cap G:s\in S_0)$ and $\mathscr V=(B(a_s,2r_s)\cap G:s\in S_0)$. We can assume that the coverings $\mathscr U$ and $\mathscr V$ of $G$ are well-ordered and let  $\mathscr U=(U_\xi:\xi\in[0,\alpha))$,  $\mathscr V=(V_\xi:\xi\in[0,\alpha))$.  Notice that $U_\xi$ and $V_\xi$ are convex subsets of $X$.

For every $\xi\in[0,\alpha)$ we fix $a_\xi\in U_\xi$. Consider a direct sum
  $$
  A=\bigoplus_{\xi\in[0,\alpha)} (\overline{V_\xi}\times\{\xi\}).
  $$
  Denote $G_\xi=U_\xi\times\{\xi\}$, $F_\xi=\overline{V}_\xi\times\{\xi\}$. Fix also any $b_\xi\in \overline{V_\xi}\setminus{V}_\xi$ for each $\xi$ and consider the following subspace of $A\times [0,1]$ with the Tychonoff topology:
  $$
  B=\bigcup_{\xi\in[0,\alpha)} (\{(b_\xi,\xi)\}\times[0,1])\cup A.
  $$
  By letting
  $$
  (a',t')\sim (a'',t'') \Leftrightarrow t'=t''=1 \quad\mbox{or}\quad a'=a''\,\,\mbox{and}\,\, t'=t''
  $$
  we define an equivalence relation $\sim$ on $B$. Then the subspace $\bigcup_{\xi\in[0,\alpha)} (\{(b_\xi,\xi)\}\times[0,1])/\sim$ is the hedgehog $J=J(|\alpha|)$.
     It is well-known~\cite[p.~251]{Eng-eng} that there exists a metric $d_J$  which generates the topology of  $J$. Clearly, the set $F_\xi$ is homeomorphic to $\overline{V_\xi}$ and let $d_X$ be a metric on $A$ induced from $X$.

     For all $b'=[(a',t')]$ and $b''=[(a'',t'')]$ from $B$ we put
  \begin{gather*}
    \varrho(b',b'')=\begin{cases}
                d_X(a',a''), & t'=t''=0\,\,\mbox{and}\,\, a',a''\in F_\xi,\\
                d_J(b',b'',), & b',b''\in J,\\
                d_X(a',(a_\xi,\xi))+d_J([(a_\xi,0)],[(a'',t'')]), & a'\in F_\xi, b''\in J,\\
                d_X(a',(a_\xi,\xi))+d_J([(a_\xi,0)],[(a_\eta,0)])+d_X((a_\eta,\eta), a''), & a'\in F_\xi, a''\in  F_\eta.
                  \end{cases}
  \end{gather*}
 It is easy to see that the  metric $\varrho$ generates the topology on $B$. We put
 $$
 Y=(B,\varrho)
 $$
 Then $Y$ is a contractible and locally contractible metric space.
 Let us observe that $(G_\xi:\xi\in[0,\alpha))$ is a discrete family of open sets in $Y$.

 Fix $\xi\in[0,\alpha)$ and $x_0\in G$. Since $\overline{U_\xi}$ is convex, we may use Tietze Extension Theorem and take a continuous map $f_\xi:Y\to \overline{U_\xi}$ such that $f_\xi(x,\xi)=x$ for all $x\in U_\xi$. Clearly, $f_\xi|_{G_\xi}$ is a homeomorphism.
 Since $G$ is path-connected, there exists a homeomorphic embedding \mbox{$\gamma_\xi:[0,1]\to G$} such that $\gamma_\xi(0)=a_\xi$ and $\gamma_\xi(1)=x_0$. Notice that the set $\overline{U_\xi}\cup\gamma_\xi([0,1])$ is contractible. Therefore, by Lemma~\ref{lem:ad_contr} there exists a continuous map $g_\xi:Y\to \overline{U_\xi}\cup\gamma_\xi([0,1])$ such that $g_\xi|_{\overline{G_\xi}}=f_\xi|_{\overline{G_\xi}}$ and $g_\xi|_{Y\setminus W_\xi}=x_0$.

Now for all $y\in Y$ we set
\begin{gather*}
  \varphi(y)=\left\{\begin{array}{ll}
                      g_\xi(y), & y\in W_\xi,\\
                      x_0, & y\not\in\bigcup_{\xi\in[0,\alpha)}W_\xi.
                    \end{array}
  \right.
\end{gather*}
It is easy to see that the map $\varphi:Y\to G$ is a weak local homeomorphism.
\end{proof}

\begin{theorem}\label{th:X_YANR}
Let $X$ be a topological space and $Y$ be a metrizable ANR. Then
\[ \cB{X,Y}=h\cB{X,Y}. \]
Moreover, every $f\in\cB{X,Y}$ is a pointwise limit of a sequence of null homotopic continuous maps.
\end{theorem}

\begin{proof}
We only need to prove the inclusion $ \cB{X,Y}\subseteq h\cB{X,Y}$.

Let $f\in \cB{X,Y}$.
The Kuratowski-Wojdys{\l}awski theorem~\cite[Theorem III.8.1]{BorsukTR} implies that $Y$ is a closed subset of a convex set $E$ in a Banach space.
Then there exist an open subset $G$ of $E$ and a retraction $r:G\to Y$.
According to Theorem~\ref{thm:cover}, $G$ is weakly covered by a metrizable contractible and locally contractible space.
Corollary~\ref{cor:cover} implies that there exists a continuous map $H:X\times [0,+\infty)\to G$ such that $\lim_{n\to\infty} H(x,n)=f(x)$ for all $x\in X$.
Then $r\circ H:X\times [0,+\infty)\to Y$ is a $\Bo$-homotopy for~$f$.
Moreover, since $G$ is contractible, each map $(r\circ H)_t:X\times t \to Y$ is null-homotopic.
\end{proof}

Proposition~\ref{ex:b1_is_not_hb1} and Theorem~\ref{th:X_YANR} imply the following question.
\begin{question}
  Do there exists  a   path-connected subset $X\subseteq\mathbb R^2$ such that $\cB{X,X}\ne \hcB{X,X}$?
\end{question}

\end{document}